\documentclass[12pt]{amsart}
\usepackage{amssymb,amsmath,amsthm}
\newtheorem{theorem}{Theorem}

\newtheorem{proposition}[theorem]{Proposition}
\newtheorem*{definition}{Definition}

\newtheorem*{atf1}{\bf Theorem [Conjecture of Andrianov and Panchishkin]}
\newtheorem*{atf2}{\bf Theorem[Conjecture of Deligne, Yoshida, Panchishkin]}
\theoremstyle{remark}

\newtheorem*{remark}{Remark}

\numberwithin{theorem}{section} \numberwithin{equation}{section}

\newcommand{\A}{\mathbb{A}} 

\newcommand{\R}{\mathbb{R}}
\newcommand{\C}{\mathbb{C}}

\newcommand{\Q}{\mathbb{Q}}

\newcommand{\Z}{\mathbb{Z}}
\newcommand{\N}{\mathbb{N}}
\newcommand{\cM}{\mathcal{M}}

\def\H{{\mathfrak H}}

\begin{document}
\setlength{\parindent}{0cm}
\sf              
\title[Towards Functoriality 
Of Spinor L-Functions]
{Towards Functoriality Of Spinor L-functions}
\author{Bernhard Heim}
\address{Max-Planck Institut  f\"ur Mathematik, Vivatsgasse 7, 53111 Bonn, Germany}
\email{heim@mpim-bonn.mpg.de}
\begin{abstract}
The object of this work is the spinor L-function of degree $3$ and certain degeneration
related to the functoriality principle. We study liftings of automorphic forms on the pair of symplectic groups
$\left(\text{GSp}(2),\text{GSp}(4)\right)$ to $\text{GSp}(6)$. 
We prove cuspidality and demonstrate the compatibility with conjectures of 
Andrianov, Panchishkin, Deligne and Yoshida. This is done on a motivic and analytic level.
We discuss an underlying torus and L-group homomorphism and put our results in the context
of the Langlands program.
\end{abstract}
\maketitle
\section{Introduction}
The {\it principle of functoriality} describes deep 
conjectural relationships among automorphic representations on
different groups \cite{La97}. 
In this paper we consider automorphic representations and liftings of the pair of
groups $\text{GL}(2)$, $\text{GSp}(4)$ to $\text{GSp}(6)$, where
$\text{GSp}(2n)$ denotes the symplectic group with similitudes of degree $n$.
The lifting will be described in terms of automorphic representations
$\pi \in \mathcal{A}(\text{GSp}(2n))_k$
with holomorphic discrete series at archimedean places (of weight $k$),
and convolutions of their spinor L-functions.
This leads to map
\begin{equation}
\mathcal{A}(\text{GL}(2)(\A)_{k-2} \times \mathcal{A}(\text{GSp}(4)(\A)_{k} 
\longrightarrow \mathcal{A}(\text{GSp}(6))(\A)_{k}.
\end{equation}
Here $\A$ denotes the ad\'{e}les over $\Q$.
We prove a cuspidality criterion and show that all
other formally possible weights can not occur.
Moreover we prove that the lifting is compatible with
conjectures of Andrianov \cite{An74}, Yoshida \cite{Yo01}, 
Panchishkin \cite{Pa07} and Deligne \cite{De79} relating the functional equation
of the spinor L-function of degree $3$ and the arithmetic nature of the critical values.
Morover the underlying $L$-function can be identified with
the spinor L-function of
\begin{equation}
\text{GL}(2)(\A) \times \text{GSp}(4)(\A).
\end{equation}
An integral representaion, functional equation and special values results have been indepently
obtained by Furusawa \cite{Fu93} and Heim \cite{He99}, and B\"ocherer and Heim \cite{BH00}, \cite{BH06}.
The results and observations obtained in this paper give strong evidence
that the Miyawaki conjecture \cite{Mi92}, \cite{He07d} of type II
is essentially not only the lift of a pair of two elliptic modular forms,
in contrast, it can be viewed as a lifting of a pair of an elliptic modular form
and a Siegel modular form of degree $2$. 
There is also promising numerical
data due to Miyawaki\cite{Mi92} supporting this approach. 
Let $\lambda_p(F)$ be the $p$-th Hecke eigenvalue of Siegel cuspidal Hecke eigenform.
By a tremendous effort Miyawaki succeded in determine the second eigenvalue of 
the Hecke eigenform $F_{14}$ of degree $3$ and weight $14$:
$\lambda_2(F_{14}) = -2^7 \cdot 2295$.
In this interesting case our method gives the simple statement that
\begin{equation}
\lambda_2(F_{14}) = \tau(2) \cdot \lambda_2(G),
\end{equation}
where $\tau(2)=-24$ is 
the second Fourier coefficient of the Ramanujan $\Delta$-function
and $G$ the non-trivial Siegel cuspform of weight $14$ and degree $2$.
Experts in the this field, as Ibukiyama, Poor and Yuen,
are optimistic that such calculations could be done
for higher weights in the near future.
We also would like to mention the recent work 
of Chiera and Vankov \cite{CV08},
which is the first in a series of papers
attacking this problem by starting with the weight of first occurence $k=12$,
in which the conjecture of Miyawaki is now fully proven by Ikeda \cite{Ik06} and Heim \cite{He07d}.
\\
\\
Concerning numerical data, Langlands [\cite{La79a},p. 213] stated: 
\begin{center}
\it
The necessary equalities are too complicated to be merely coincidences,\\ and we may assume with some
confidence, ... .
\end{center}
This was in reference to the paper of Kurokawa \cite{Ku78}, where a
statement of what is now called the Saito-Kurokawa lifting had been
verified for small primes, broaching what has become an important
research subject over the last $30$ years. 
The underlying principle of our lifting can be considered as a
higher dimensional analog of the $L$-group homomorphism 
:
\begin{equation}
\text{GL}(2)(\C) \times \text{GL}(2)(\C) \longrightarrow \text{GL}(4)(\C)
\end{equation}
given by Ramakrishnan \cite{Ra00}, which goes beyond endoscopic considerations.
To explain this will require some preparation.
It will become apparent that the approach here is only the tip of an
iceberg, suggesting optimism about further results in this
direction. We note that not much is known concerning functoriality and
lifting of automorphic representations without the assumption of a
Whittaker model. At this point we would like to suggest the reader to
consult the excellent overview article of Raghuram and Shahidi \cite{RS08}, and
especially chapter 7, with the desription of Harders fundamental work \cite{Ha83}.
\\
\\
Let $G/K$ be a reductive group over a number field $K$ and
$\mathcal{A}(G)$ the space
of automorphic representations $\pi$ of $G(\mathbb{A}_K)$, where $\mathbb{A}_K$ are the adeles of $K$.
A general question addressed in the Langlands program \cite{Ge84} arises 
if one takes the convolution L-function of two automorphic representations, 
and asks whether there exists another automorphic representation with
this L-function. 
Let $G_1,G_2$, $G_3$ be three 
reductive groups defined over the same number field $K$.
Let $\rho_1,\rho_2$ be finite dimensional representations of the $L$-groups of $G_1$ and $G_2$.
In the sense of 
Langlands (cf. \cite{GS88}) we attach to 
the data $(\pi, \rho)$ the {\it Langlands L-function} $L(s,\pi,\rho)$.
In this notation we can ask whether there exists a map
\begin{equation}
\boxtimes \,\,\,\,
\begin{cases}
\mathcal{A}(G_1) \times \mathcal{A}(G_2) &\longrightarrow \quad \mathcal{A}(G_3),\\
\left( \pi^1, \pi^{2}\right) & \mapsto \quad \quad \pi^1 \boxtimes \pi^{2},
\end{cases}
\end{equation}
with an automorphic representation 
$\pi:= \pi^1 \boxtimes \pi^{2} \in \mathcal{A}(G_3)$ which has the property
that the convolution L-function $L(s,\pi^1 \otimes \pi^{2},\rho_1 \otimes \rho_2)$ is equal to
a L-function $L(s,\pi^1 \boxtimes \pi^{2},\rho_3)$, 
where $\rho_3$ is a suitable finite dimensional representation of
the $L$-group of $G_3$.
One is far away from having a general solution of this problem and it seems 
that even the {\it first examples} given in the literature demonstrate the depth
of the issues.
This is illustrated by the work of Ramakrishnan 
\cite{Ra00} and Kim and Shahidi \cite{KS02}.
Here we briefly recall results of Ramakrishnan.
He constructed a map
\begin{equation}
\mathcal{A}(Gl_2) \times \mathcal{A}(GL_2) \longrightarrow \quad \mathcal{A}(GL_4),
\end{equation}
where $\rho_i$, $1 \leq i \leq 3$ is the standard representation.
The existence of such a map had been expected for many years, but
several difficult auxiliary results were needed. Since we are in a
parallel situation, we mention some essential auxiliary results still
needed in order to prove the Main Conjecture in \cite{He07e}.
We recall briefly some of the main ingredients in the proof. To quote
Ramakrishnan \cite{Ra00}: 
\begin{center}
\it
Our proof uses a mixture of converse theorems,
base change and descent, and it also appeals to the local 
regularity properties of Eisenstein series and the scalar products of their truncation.
\end{center}
This involves a converse theorem of Cogdell and Piatetski-Shapiro \cite{CP96}, 
base change results of Arthur and Clozel \cite{AC89}, 
results of Garrett \cite{Ga87} and Piateski-Shapiro 
and Rallis \cite{PR87} towards the triple product L-function and
finally applications of Arthur's truncation \cite{Ar80} 
of non-cuspidal Eisenstein series to get boundness in vertical strips of the 
triple product L-functions.
Recently the L-functions considered in this paper showed up 
several interesting papers.
In dissertations of Agarwal \cite{Ag07} and Saha \cite{Sa08} the level aspect and $p$-adic properties
had been considered.
Both have beautiful applications in mind. The first indicates
applications towards functoriality \`{a} la Ramakrishnan and
the second one goes towards the modularity approach of the
Iwasawa conjecture related to ideas of Skinner and Wiles.

In two manuscripts Pitale and Schmidt \cite{PS1},\cite{PS2}
consider the L-function of $\text{GSp}(4) \times \text{GL}(2)$ mainly for holomorphic modular forms
and partly refine the results of Furusawa \cite{Fu93} and Heim \cite{He99},
and B\"ocherer and Heim \cite{BH00}, \cite{BH06} in the level aspect.

Hence it is obvious that considering the $\text{GSp}(4) \times \text{GL}(2)$
spinor L-function for the family of weights $k,k-2$, even in the 
highly assumed case of endoscopy (e.g. Saito-Kurokawa lifts)
gives a significant new approach to several conjectures related to
the spinor L-function of Siegel modular forms of degree $3$.

Recently Miyawaki's conjecture for $\text{F}12$ had been proven \cite{He07d}.
Applications are indicated by Dalla Piazza and van Geemen \cite{PG1}.
These L-series and their relation to Galois representations are of great interest
in mathematical physics, where they show up in the bosonic and superstring measures.

\section*{Acknowledgements}
This paper was supported by the Max Planck Institute of Mathematics
in Bonn during my stay between Mai/July 2008.
I also would like to thank Christian Kaiser for
several very useful discussions on the topic.

\section{Summary of the main results}
The lifting we predict would be an example of Langlands'
functoriality principle. Since our predicted lift goes beyond endoscopy, and since it
satiesfies all known motivic and analytic viewpoints and which
is compatible with conjectures of Andrianov,Deligne, Yoshida, Panchiskin,
we have to generalize the concept of $L$-group homomorphism.

The L-group of the group of projective symplectic similitudes of degree $n$ is
the spin group $\text{Spin}(2n+1)(\C)$.
Let $r_n$ be complex representations of the L-groups
$\text{Spin}(2n+1)(\C)$. Langlands conjectures that for any L-group homomorphism
\begin{equation}
\Theta:
\text{Spin}(3)(\C) \, \times \, \text{Spin}(5)(\C)  \longrightarrow  \text{Spin}(7)(\C)
\end{equation}
there exists a map
\begin{equation}
\Theta^*: \mathcal{A}(PGSp(2)) \times \mathcal{A}(PGSp(4)) \longrightarrow \mathcal{A}(PGSp(6)) 
\end{equation}
such that the associated L-functions are equal, namely,
\begin{equation}
L(s,\Theta^*(\pi^1,\pi^2),r_3) = L(s,\pi^1 \otimes \pi^2,r_1 \otimes r_2).
\end{equation}
Let $\rho_n$ be the spin representation of $\text{Spin}(2n+1)(\C)$.
Since we are interested in a reasonable generalization of the Miyawaki
conjecture, and in a proof of Andrianov's and Deligne's conjecture for such
liftings for the spinor L-function of degree $3$, we restrict ourself
to the space  $\mathcal{A}(PGSp(2n))_k$
of all automorphic representations
$\pi = \pi_F$, associated to Siegel modular forms $F$ of degree $n$ and weight $k$.
Take $r_i =\rho_i$ (for $n=1,2,3$).
To study the spinor L-function of degree $3$ 
fix a maximal torus $T(n)$ in $\text{Spin}(2n+1)$ parametrized by
$(a_0,a_1,\cdots, a_n)$ with $a_i \in \C^\times$.
We take the torus homomorphism $\Theta: T(3) \times T(5) \longrightarrow T(7)$,
or lifted L-group homomorphism, given by
\begin{equation}
\Theta
\Big((\alpha_0, \alpha_1) , (\beta_0,\beta_1,\beta_2) \Big)\;=\;
(\alpha_0 \beta_0,\,\beta_1,\,\beta_2,\,\alpha_1).
\end{equation}
We are aware that this seems to be very bizarre, but since
we can not assume endoscopy in general (since we not only consider 
Saito-Kurokawa cuspforms), it seems that this
is the only way to state such a lift in the framework
of Langlands description of lifts in terms of Satake parameters.

Now suppose that the corresponding functorial map $\Theta^*(\pi^1,\pi^2)$
exists. Then put 
\begin{equation}
\pi:= \pi^1 \boxtimes \pi^2 := \Theta^*(\pi^1,\pi^2).
\end{equation}
{\it A priori} it is not clear that these conjectural lifts are cuspidal. In
this paper we show that they are. Let the corresponding cuspidal Siegel
modular forms of degree $1,2,3$ have weights $k_1,k_2,k_3 \in \N$. 
We prove that only the triple $k-2,k,k$ can occur. This conclusion includes
Miyawaki's conjecture \cite{Mi92} of type II \cite{He07e}.
We also give a motivic interpretation of these lifts and deduce the same result.
Now let $F= h \boxtimes G$ with suitable $h\in S_{k-2}, G \in S_k^2, F \in S_k^3$, where $S_l^n$ is the space
of cuspidal Siegel modular forms of weight $k$ and degree $n$. Let $\lambda_p$ be the $p$-th Hecke eigenvalue.
Then, for all primes $p$, the eigenvalues of the $h,G,F$ satisfy 
\begin{equation}
\lambda_p(F) = \lambda_p(h) \cdot \lambda_p(G).
\end{equation}
We prove the conjecture of Andrianov and Panchiskin
on the meromorphic continuation and functional equation of the
spinor L-function of degree $3$ for this type of cuspidal Hecke eigenform
(see section \ref{sectionconjectures} for details).
\begin{atf1}
Let $k_1,k_2$ be positive even integers.
Let $\pi^1 \in \mathcal{A}(GL(2))_{k_1}$ and $\pi^{2} \in 
\mathcal{A}(GSp(4))_{k_2}$ 
be two automorphic representations. Let $\pi := \pi^1 
\boxtimes \pi^2\in \mathcal{A}(GSp(6))_{k}$ be modular.
Then necessarily $k_1=k-2$ and $k_2=k$.
The L-function $L(s, \pi, \rho_3)$ has a meromorphic continuation to the 
whole complex plane. 
Let the first Fourier-Jacobi coefficient of $\pi^2$ be non-trivial. Then
the completed L-function
\begin{equation}
\widehat{L}(s,\pi,\rho_3) =  L_{\infty}(s, \pi,\rho_3) \, L(s,\pi,\rho_3)  ,
\end{equation}
has a functional equation under $s \mapsto 3k-5-s$, where
$$
L_{\infty}(s,\pi) := \Gamma_{\C}(s) \, \prod_{m=1}^3 \, \Gamma_{\C}(s-k+m).$$
\end{atf1}
For $\pi^2$ attached to a Saito-Kurokawa lift the spinor L-function
is entire. 
The special values of the lifts are apparently compatible with
Deligne's conjecture on the arithmetic of 
special values specialized to spinor L-function and
to twisted spinor L-function by Yoshida and Panchishkin.
\begin{atf2}
Let $\widetilde{\pi}= \pi' \boxtimes \pi'' \in \mathcal{A}(\text{GSp}(6))_k$ for $\pi' \in \mathcal{A}(\text{GL}(2))_{k-2}$ 
and $\pi'' \in \mathcal{A}(\text{GSp}(4))_{k}$ primitive. 
Let $F_{\widetilde{\pi}} \in S_k^3$ be normalized.
Then the critical integers of the spinor L-function of $F_{\widetilde{\pi}}$ are
exactly $m=k,k+1,\ldots, 2k-5$.
Moreover there exists a positive $\Omega \in \R$ such that
\begin{equation}
\frac{L(m, F_{\widetilde{\pi}},\rho_3)}{{\pi}^{4m-3k+6} \Omega} \in E.
\end{equation}
\end{atf2}
It is surprising that the spinor L-function attached
to a Siegel modular form $F= h \boxtimes G \in S_k^3$
has a Rankin-Selberg integral representation with respect to an
Eisenstein series of Siegel type 
of degree $5$ \cite{He99}, \cite{BH00}, \cite{BH06}.

\section{Basic notation and facts}
In this section we describe the relation between
Siegel modular forms
and automorphic representations.
We use the approach of Langlands
to systematically define the spinor, standard, and $GL(2)$-twisted
spinor L-function.
\\
\\
Let $k,n$ be positive integers and let
$S_k^n=S_k(\Gamma_n)$ be the space of cuspidal Siegel modular forms of weight $k$
with respect to the Siegel modular group $\Gamma_n = Sp(2n)(\Z)$.
Let $ G_{2n}:=GSp(2n)$ be the symplectic group with similitudes and let
$\overline{G}_{2n}:= G/ Z$, where $Z$ is the center of $G$.
In the setting of Siegel modular forms let $F \in S_{k}^{n}$ 
be a Hecke eigenform and $\pi$ the associated 
automorphic representation of $\overline{G}_{2n}(\mathbb{A}_{\Q})$ over $\Q$, where
$\mathbb{A}_{\Q}$ is the adele ring of $\Q$.
Then $\pi$ factors over primes as $\pi = \bigotimes_v'\pi_v$, 
where $\pi_{\infty}$ is a holomorphic discrete series representation
(or {\it limit} of discrete series)
and, for finite $v=p$, $\pi_p$ is a spherical representation of
$Sp(2n)(\Q_p)$. Here $p$ are the prime numbers in $\Q$. 
The local representations are uniquely determined by the {\it Satake parameters } of $F$:
$$\mu_{0,p}^F, \mu_{1,p}^F \cdot
\ldots \cdot \mu_{n,p}^F.$$ 
These Satake parameters are unique up to the action of the Weyl group of $G_{2n}$.
There is a clear correspondence between the eigenvalues of the Hecke
operators in the classical  
setting and the Satake parameters coming from the local representations.
For example let $\lambda_p(F)$ the eigenvalue of $F$ with respect to the Hecke operator $T_p^{(n)}$.
Here $T_p^{(n)}$ is the canonical generalization of the well-known Hecke operator $T_p$
for elliptic modular forms on the upper half-space $\H:=\left\{ z = x + iy \vert \, y>0 \right\}$.
Then by standard normalization we have
\begin{equation}
\lambda_p(F) = \mu_{0,p}^F \left( 1 + \sum_{r=1}^n \prod_{1\leq i_1 \ldots 
\leq i_r} \mu_{i_1,p}^F \cdot \ldots \cdot \mu_{i_r,p}^F\right).
\end{equation}
To fix notation, let $H$ be any linear algebraic reductive group (split) over
a number field $K$. Then
\begin{equation}
\mathcal{A}(H) := 
\left\{ \pi 
\text{ cuspidal automorphic representation of } 
H(\A_{K} )
\right\}.
\end{equation}
Let $H$ be the projective group of symplectic similitudes and $k$ a  positive integer.
Then we denote by $\mathcal{A}(H)_k$ the set of all representations $\pi$ in
$\mathcal{A}(H)$ such that there exists a $F \in S_k^n$ with $\pi = \pi_F$.

Now we can reformulate our lifting problem. This leads to the following
necessary condition. Let 
\begin{itemize}
\item
$\pi^1 \in \mathcal{A}(PGSp(2)(\A))_{k_1}$ 
\item
$\pi^2\in \mathcal{A}(PGSp(4)(\A))_{k_2}$. 
\end{itemize}
Then there has to exist a
$\pi^3\in \mathcal{A}(PGSp(6)(\A))_{k_3}$ such that for all finite primes $p$:
\begin{equation}
\fbox{$\lambda_p(F) = \lambda_p(h) \cdot \lambda_p(G)$}.
\end{equation}
Here $h,G,F$ are related to $\pi^1,\pi^2,\pi^3$ as above.
Now we want to put this observation into the picture of the Langlands program
to apply standard techniques and obtain possible generalizations.

\subsection{Automorphic L-functions}
Langlands' set-up attaches automorphic L-functions 
to reductive algebraic groups in a systematic way and recovers the
classical L-functions as the Hecke L-function, Rankin-Selberg
L-function and Andrianov Spinor L-function. 
Let us recall the main ingredients. 
We fix the following data
related to a linear algebraic group (split) over some number field $K$.

\begin{itemize}
\item
$G(\A_K)$ denote the restricted direct product $\prod_{v} G(K_v)$, here the $K_v$
are the local fields attached to the valuations $v$.
\item
$\!\!\!  \!\!\!\phantom{X}^L G$ the $L$-group of $G$, which is a well-defined subgroup of $GL(n)(\C)$,
$n$ suitable.
\item
$\rho$ a finite dimensional representation of the $L$-group of $G$.
\item
$\pi$ an automorphic representation of $G(\A_F)$, $\pi = \otimes \pi_v$.
\end{itemize}

For example let $G= PGSp(2n)$, then the L-group
coincides in this
special case with the universal covering $Spin(2n+1)(\C)$ of the orthogonal group $S0(2n+1)(\C)$.
We denote by $\rho_n$ and $st_n$ the spin representation 
of dimension $2n$ and the standard imbedding into $GL (2n+1)(\C)$ of the L-group.

Let $\pi \in \mathcal{A}(G)$ with $\pi = \otimes_p \pi_v$, 
then for almost all $v$ the local representations
$\pi_v$ are spherical and unramified and correspond 
to the conjugacy class of a semi-simple element $t_v^{\pi}$ in the L-group. Let $S$
be a finite set, which contains all places at infinity and all finite $v$, which are
ramified.
Then the Langlands L-function attached to this data is defined by
\begin{equation}
L^{(S)}(s,\pi,r):= 
\prod_{v \not\in S} 
\text{det} \left( 
1 - r(t_v^{\pi})  (Nv)^{-s}
\right)^{-1}.
\end{equation}
Here $Nv$ is the order of the residue field related to $v$.
If we work over $\Q$ and $S=\{\infty\}$ we skip the $S$.
This $L$-function is holomorphic with respect to $s \in \C$ for $\text{Re}(s)$ large enough.
Langlands conjectured that the $L$-function can be completed such that
one gets a meromorphic function on the whole $s$-plane with a functional equation.

%

\begin{definition}
Let $n,k \in \N$ and $\pi \in \mathcal{A}(GSp(2n))_k$. Then the {\it spinor} and {\it standard}
L-functions are given by
\begin{eqnarray}
L(s, \pi, \rho_n) & := 
\prod_p 
\left( 
\prod_{k=0}^n 
\prod_{1 \leq i_1 < \ldots < i_k \leq n} 
\left( 1- \mu_0 \mu_{i_1} \ldots \mu_{i_k} \, p^{-s}\right)
\right)^{-1}\\
L(s, \pi, st_n) & := 
\prod_p 
\left( (1-p^{-s} ) \prod_{j=1}^n 
\left( 1- \mu_j \, p^{-s}\right) \, 
\left( 1-\mu_j^{-1} \, p^{-s} \right) 
\right)^{-1}.
\end{eqnarray}
\end{definition}
Let $\pi= \pi_{\infty} \otimes_p \pi_p \in \mathcal{A}\left( \text{GSp}(2n)\right)$. Then $\pi$ satisfies the
Ramanujan-Petersson conjecture at finite places if for all primes $p$ and local finite representations $\pi_p$
the Satake parameters $\mu_0, \mu_1, \ldots, \mu_n$ satisfy
\begin{equation}
\vert\mu_1\vert = \vert\mu_2\vert =  \ldots =  \vert\mu_n\vert = 1.\\
\\
\end{equation}
{\it \bf Two remarks:}\\
a) The spinor and standard L-function converges absolutely and locally uniformly for real part of $s$ large enough.
For example without any assumption one can choose
$\text{Re}(s) > n+1$ in the case of the standard L-function.
\\ 
b) Let $\pi \in \mathcal{A}\left( \text{GSp}(2n)\right)_k$ satisfy the Ramanujan-Petersson conjecture, then the
Euler products converge already for 
\begin{equation}
\text{Re}(s) > \Big( nk - n(n+1)/2\Big)/2 \text{ and }\text{Re}(s) > 1. 
\end{equation}
\ \\
It is well known that the Ramanujan-Petersson 
conjecture is fullfilled in the case of elliptic cusp forms.
In the case $n=2$ the Saito-Kurokawa lifts (CAP-representations) contradict the conjecture.
Nevertheless for all other representations $\pi \in  \mathcal{A}\left( \text{GSp}(4)\right)_k$
the Ramanujan-Petersson conjecture is satisfied \cite{We94}, \cite{We05}. \\
Let $c_1:= (k+1)/2$, $c_2 := k -1/$, and $c_3 := (3k)/2 -2$.
Then we have the sharp bounds $\text{Re}(s) > c_n$ for the spinor L-functions in the cases $n=1,2,3$, 
if the conjecture is satisfied.

\subsection{Conjectures of Andrianov and Deligne}\label{sectionconjectures}
Let $ \pi \in \mathcal{A}\left( \text{GSp}(6)\right)_k$. 
Then there are several open conjectures on the analytic and
arithmetic properties of the spinor L-function attached to $\pi$. 
\\
\\
{\bf Functional equation [{\it Conjecture of Andrianov \cite{An74}, \cite{Pa07} }]}\\
The completed spinor L-function $\widehat{L}(s, \pi, \rho_3)$ attached to
an automorphic representation $\pi \in \mathcal{A}\left( \text{GSp}(6)\right)_k$ has a
meromorphic continuation to the whole complex plane and satisfies the functional equation
$$
s \mapsto 3k -5 -s.$$
The completion is given by
\begin{equation}
\widehat{L}(s, \pi, \rho_3) := L_{\infty}(s,\pi) \, L(s, \pi, \rho_3),
\end{equation}
where $\Gamma_{\C}(s) := 2 (2 \pi)^{-s} \Gamma(s)$ and
\begin{equation}
L_{\infty}(s, \pi) :=\Gamma_{\C}(s) \, \prod_{m=1}^3 \, \Gamma_{\C}(s-k+m).
\end{equation}
\
\\
\\
{\bf Arithmetic of critical values [{\it Conjecture of 
Deligne \cite{De79}, Panchiskin \cite{Pa07}, Yoshida \cite{Yo01}}]}\\
Let $\mathcal{M}$ be a (hypothetical) motive 
attached to $\pi \in \mathcal{A}\left(\text{GSp}(2n)\right)_k$ with
$L$-function $L(s,\pi,\rho_n)$. 
Hence $\mathcal{M}$ is a motive over $\Q$ with coefficients in an algebraic number field $E$.
Let $R:= E  \otimes_{\Q} \C$ with $E \subset R$.
Then the L-function $L(s, \mathcal{M})$ of the motive takes values in $R$. 
Let $m \in \Z$ be a critical value,
then Deligne's conjecture predicts that
\begin{equation}
\frac{L(m, \mathcal{M})}{(2 \pi \, i)^{d^{\pm} m} \, c^{\pm}(\mathcal{M}) } \in E.
\end{equation}
Here $\pm$ has the same sign as $(-1)^m$ and $c^{+}(\mathcal{M}), 
c^{-}(\mathcal{M})$ are Deligne's periods.
The natural integers $d^{+}$ and $d^{-}$ are the dimension of the 
$+$ and $-$ eigenspace of the Betti realization of $\mathcal{M}$.
Let $n=3$ then Panchiskin \cite{Pa07} has determined explicit the 
interval of the involved critical values. They are given by the
positive integers $k, k+1, \ldots, 2k-5$.
Since we prove that the period does not depend on
the parity of the special values we ommit the discussion on
the meaning of this fact in this paper. We only want to note that this
somehow seems to reflect the degeneration of the lifted Siegel modular form of
degree $3$ on the level of periods. Of course one could speculate and
ask if this already determine lifts in the case of Siegel modular forms of degree $3$.

\section{Candidates for the lifting - first properties}

\begin{definition}
Let $h \in S_{k_1}^{n_1}$ and $G \in S_{k_2}^{n_2}$ be Hecke eigenforms.
Let $\pi^1$ and $\pi^2$ be the associated automorphic representations and let $r_1$ and $r_2$
be two complex respresentations of dimension $m_1$ and $m_2$ of the corresponding $L$-groups.
Then 
\begin{equation}
L(s, \pi^1 \otimes \pi^2, r_1 \otimes r_2):=
\prod_p \left( 1_{m_1 + m_2} - r_1(t_p^{\pi^1}) \otimes r_2(t_p^{\pi^2}) \, p^{-s}\right)^{-1}.
\end{equation}
If there is any automorphic representation $\pi^3$ (and complex representation $r_3$ of the related $L$-group)
such that
\begin{equation}
L(s,\pi^3,r_3)=L(s, \pi^1 \otimes \pi^2, r_1 \otimes r_2),
\end{equation}
then we put $\pi^3 = \pi^1 \boxtimes \pi^2$ to indicate the modularity of the tensor product.
\end{definition}
Already in the work of Ramakrishnan \cite{Ra00} it was a non-trivial task to show
that the possible lifting is forced to be cuspidal. Now since our groups have higher rank
one can expect that this is the same case here also. 
In contrast to the $GL_n$ case the Ramanujan-Petersson conjecture fails in general.
Let $M_k^n$ be the space of Siegel modular forms of degree $n$, weight $k$ with respect to
the Siegel modular group $\Gamma_n$.

\begin{theorem}
Let $\pi= \pi(F)$ be a holomorphic automorphic representation of $\text{GSp}(6)(\mathbb{A})$
associated to a Hecke eigenform $F$ in $M_k^3$.
Assume that $\pi= \pi^1 \boxtimes \pi^2$ is the lifting of the two cuspidal automorphic representations
$\pi^1 \in \mathcal{A}\left(GL(2)\right)_{k_1}$ 
and $\pi^2\in \mathcal{A}\left(GSp(4)\right)_{k_2}$ 
of even weights $k_1,k_2$. Then $\pi$ is cuspidal.
\end{theorem}
In the case of the lifting of $\text{GL}(2) \times \text{GL}(2)$ to $ \text{GL}(4) $
this was called by Ramakrishnan {\it cuspidality criterion}, and employed in \cite{Ra00}, section 3.2.

\begin{proof}
Let $\pi^1 \in \mathcal{A}\left(GL(2)\right)_{k_1}$ 
and $\pi^2 \in \mathcal{A}\left(GSp(4)\right)_{k_2}$ 
be two automorphic representations with even weights $k_1,k_2$. Let us assume that
an automorphic representation $\pi$ exists coming 
from some Siegel eigenform $F$ of degree $3$ of weight $k_3$, 
such that $$L(s,\pi,\rho_3) = L(s, \pi^1 \otimes \pi^2, \rho_1 \otimes \rho_2).$$
Let us recall a result of Chai and Faltings [\cite{CF90}, p. 107].

\begin{proposition}[Chai-Faltings]
Let $F \in S_k^n$ be a Hecke eigenform. Let $k>n$. Then we have for every finite prime $p$, that 
there exists at least one element in the Weyl group
orbit of the $p$-Satake parameters $\mu_0, \mu_1, \mu_2, \ldots, \mu_n$ such that
\begin{equation}
\vert \mu_1 \cdot \mu_2 \cdot \ldots \cdot \mu_n \vert =1.
\end{equation}
Here we have the normalization $\mu_0^2 \mu_1 \ldots \mu_n = p^{nk - n(n+1)/2}$.
\end{proposition}

Let $\alpha_0, \alpha_1$ and $\beta_0,\beta_1,\beta_2$ 
be the Satake parameters of $\pi_p'$ and $\pi_p''$.
Here we have choosen the normalization 
$$\alpha_0^2 \alpha_1 = p^{k_1-1}, \quad 
\quad \beta_0^2 \beta_1 \beta_2 = p^{2k_2 -3}.$$
If we assume that $\pi$ is not cuspidal, then it follows from the Darstellungssatz of Klingen and \cite{Ha81}
that $F$
is a 
Siegel Eisenstein series
or a Klingen Eisenstein series. The case $k=4$ 
is treated as the case of Siegel type Eisenstein series,
although we have Hecke summmation.
Let $\mu_0, \mu_1,\mu_2,\mu_3$ be the Satake parameter of 
$\pi_p$ with $\mu_0^2 \mu_1 \mu_2 \mu_3 = p^{3k-6}$.
From the lifting assumption we can deduce that
\begin{equation}\label{lift_parameters}
\mu_0 = \alpha_0 \beta_0, \, \mu_1 = \beta_1,\, \mu_2=\beta_2, \text{ and } \mu_3 = \alpha_1.
\end{equation}
Hence at least one of the Satake parameter $\mu_1, \mu_2,\mu_3$ of $\pi_p$ has the absolute value one.
We now proceed case by case.
\\
\\
a) 
Let first $F$ be a Siegel Eisenstein series. We can assume that the weight $k$ is largen then $3$.
It is well known that the Satake parameters $\mu_1,\mu_2,\mu_3$ have the values
$p^{k-3},\, p^{k-2}, p^{k-1}$. This is unique up to the 
action of the Weyl group of the symplectic group.
Since none of the values has absolute value one we have a contradiction.
\\
b)
Next we assume that $F$ is a Klingen Eisenstein series attached to the cusp form $H \in S_k^2$.
Here $H$ is also a Hecke eigenform. Let $\gamma_0,\gamma_1,\gamma_2$ be the $p$-th Satake parameter of $H$.
Then it follows from a Theorem of Chai and Faltings, that these 
parameters can be choosen in such a way, that
$\vert \gamma_1 \gamma_2 \vert =1$. Then the Satake parameters of $F$ are 
given by $\mu_1 = \gamma_1, \, \mu_2 = \gamma_2, \, \mu_3 = p^{k-3}$. From (\ref{lift_parameters})
we know that (up to the action of the Weyl group) at least one of the 
Satake parameters $\mu_1,\mu_2,\mu_2$ has
absolute values $1$. But if this would be the case for the Klingen Eisenstein series attached to $H$, then
it would follow together with the Theorem of 
Chai and Faltings that $\vert \mu_1 \vert = \vert \mu_2 \vert$. 
Hence $\vert \mu_1 \, \mu_2 \mu_3\vert$ is always equal to
$p^{k-3}$ or $p^{3-k}$, hence a contradiction to the Theorem of Chai and Faltings in degree $3$.
\\
c) Now let $F$ be a Klingen Eisenstein series attached to 
a Hecke eigenform $f \in S_k$ with $p$-th Satake parameters
$\gamma_0, \gamma_1$, where $\vert \gamma_1 \vert =1$. 
Then the Satake parameters $\mu_1, \mu_2, \mu_3$ of $F$ can be choosen by
$$
\mu_1 = \gamma_1, \,\, \mu_2 = p^{k-2}, \,\, \mu_3 = p^{k-3}.
$$
But this is a contradiction.
\end{proof}
There is another possibility to prove this result by analytic methods instead of
our algebraic proof. This is done by applying the method given in
\cite{Ha81} to the properies of the standard L-function.

\section{Motivic viewpoint}

In this section we consider the lifting problem from the motivic point of view.
We determine the possible weights for the candidates and show that this is compatible
with our previous results obtained in \cite{He07e}.

We examine the possibility of the decomposition of a motive $\cM$ over $\Q$ into
the tensor product of certain motives $\cM_1$ and $\cM_2$, which are attached
to automorphic forms:
\begin{equation}
\cM = \cM_1 \otimes_{\Q} \cM_2.
\end{equation}

\subsection{Motivic decompositions}
In 1999 we first considered the opposite question, 
we started with two {\it constructed} motives $\cM_1, \cM_2$ 
and had been interested on critical 
values \cite{De79} of the motive $\cM_1 \otimes_{\Q} \cM_2$
and the explicit formula for Deligne's conjecture on the 
arithmetic of the L-function evaluated at this point.
This was a natural question after finding a new integral representation of the
$\text{GL}(2)$-twisted Andrianov spinor L-function \cite{He99}.
This motivated Yoshida to work out the formula and a general procedure
which describes how invariants and periods of motives $\cM_1, \cM_2, \ldots, \cM_r$
behave by standard algebraic operations (see also Yoshida \cite{Yo01}, section 4, page 1193).
Finally main parts of Deligne's conjecture to the L-function from above had been proven
by B\"ocherer and Heim \cite{BH00}, \cite{BH06}.
Since $\cM_1$ and $\cM_2$ had been motives attached to automorphic forms 
it is an interesting problem to study the tensor product and
ask if the new motives is a motive of an automorphic form - modularity of
motives. The question if this motivic approach of the lifting considered in this
paper leads to the same results, which we obtain by concrete analytic methods,
was raised by Harder during the Japanese-German conference 
held at the MPI Mathematik in Bonn in February 2008.


We calculate the Hodge numbers to delete all {\it forbidden lifts}.
Moreover we can predict with this approach 
the functional equation of the spinor L-function attached to a Siegel modular form
of degree $3$ and the arithmetic of the critical values.
\\
\\ 
Let $k,n$ be positive integers. 
We assume that $k>n$, since we 
are mainly interested in
$n=1,2,3$. Let $\pi \in \mathcal{A}\left( \text{GSp}(2n)\right)_k$ be an
automorphic representation with 
holomorphic discrete series at infinity.

We denote by $E_{\pi}$ the totally real number field generated 
by the Hecke eigenvalues of $\pi_f := \otimes_p \, \pi_p$.
Now it is well known that we can pick a fix vector $F = F(\pi)\in S_k^n$ attached to the
representation $\pi$, such that the field generated by the Fourier coefficients is contained
in $E_{\pi}$. We denote such $F$ normalized. In the case $n=1$ this can easily obtained by choosing
$F$ to be primitive. For Siegel modular form there is no such explicit property available.

\subsection{Determination of the possible Hodge types}
Let $\cM$ be a motive over $\Q$ with coefficients in a totally real number field $E$.
Let $H_B(\cM)$ be the Betti realization of $\cM$. 
Then $H_B(\cM)$ is a vector space over $E$ and its finite
dimension $d$ is called the rank of the motive $\cM$. 
The Hodge decomposition is given by
\begin{equation}
\text{H}_\text{B}(\cM) \otimes_{\Q} \, \C = 
\bigoplus_{p,q \, \in \, \Z} \text{H}^{p,q}(\cM).
\end{equation}
Then $\text{H}^{p,q}(\cM)$ are free $R:= E \otimes_{\Q} \C$ modules. 
If $\text{H}^{p,q}(\cM) = \{ 0\}$ whenever $p + q \neq w$.
Then we denote $\cM$ a pure motive and $w$ the (pure) weight.
In the following let $\mathcal{M} =\mathcal{M}(F)$ be the motive attached
to a cuspidal Siegel Hecke eigenform with spinor representation of the
L-group. Let $h \in S_k$ then the Hodge type of $\mathcal{M}(h)$ is given by
\begin{equation}
(0,k-1) + (k-1,0).
\end{equation}
Let $G \in S_l^2$ then the Hodge type of $\mathcal{M}(G)$ is given by
\begin{equation}
(0,2l-3) + (l-2,l-1)+ (l-1,l-2) + (2l-3,0).
\end{equation}
By employing the K\"unneth formula we obtain for the tensor product
of $\mathcal{M}(h) \otimes \mathcal{M}(G)$ the Hodge type
\begin{eqnarray*}
& &(0,2l+k-4) + (l-2,l+k-2)+ (l-1,l+k-3)+(k-1,2l-3)\\
& & +(k+l-3,l-1)+(2l-3,k-1)+(l+k-2,l-2)+(k+2l-4,0).
\end{eqnarray*}
Assuming that this tensor product is isomorphic to the motive $\mathcal{M}(F)$
attached to $F \in S_K^3$, which has the Hodge type
\begin{equation*}
(0,3K-6) + ( K-3,2K-3) + ( K-2,2K-4) + (K-1,2K-5) + \text{ symmetric terms},
\end{equation*}
then we obtain $k=K-2,l=K$. This gives the desired result.
\remark
We are very much indepted to Alexei Panchiskin, for sharing some of his
ideas with us.
\remark
Its also possible to compare the Hodge numbers in the setting of
vector-valued modular forms (see \cite{BG07}, for the relevant Hodge decompositions).


\section{Analytic properties of Spinor L-functions}
In this section we prove the meromorphic continuation 
of the spinor L-function attached to
the lifting $\pi = \pi' \boxtimes \pi''$. 
We complete the L-function at infinity and
obtain the functional equation. This result is compatible 
with the conjecture of Andrianov and Panchiskin.


\begin{theorem}
Let $k_1,k_2$ be even positive integers.
Let $\pi^1 \in \mathcal{A}(GL(2))_{k_1}$ and $\pi^{2} \in 
\mathcal{A}(GSp(4))_{k_2}$ 
be two automorphic representations. Let $\pi := \pi^1 
\boxtimes \pi^2\in \mathcal{A}(GSp(6))_{k}$ be modular.
Then $k_1=k-2$ and $k_2=k$. Let $\rho_3$ be the spinor 
representation of $Spin(7)$. 
Then the L-function $L(s, \pi, \rho_3)$ has a meromorphic continuation to the 
whole complex plane. 
Let the first Fourier-Jacobi coefficient of $\pi^2$ be non-trivial. Then
the completed L-function

\begin{equation}
\widehat{L}(s,\pi,\rho_3) =  L_{\infty}(s, \pi,\rho_3) \, L(s,\pi,\rho_3)  ,
\end{equation}
satisfies the functional equation $s \mapsto 3k-5-s$. Here
$$
L_{\infty}(s,\pi) := \Gamma_{\C}(s) \, \prod_{m=1}^3 \, \Gamma_{\C}(s-k+m).$$
\end{theorem}
\begin{proof}
Let $\alpha_0,\alpha_1$ be the Satake parameters of $\pi_p'$ and $\beta_0,\beta_1,\beta_2$
be the Satake parameters of $\pi_p^2$. They correspond to the semi-simple elemens
(conjugacy classes) $t(\pi_p^1)$ and $t(\pi_p^2)$ in the related L-groups.
Let $\rho_1$ and $\rho_2$ be the spinor representations of $SL(2)(\C)$ and $Spin(5)(\C)$ then the spinor
L-function of $\pi$ is given by
\begin{equation}
L(s, \pi,\rho_3) = \prod_p \Big\{   \text{det} 
\left( 1_8 - \rho_1\left(t(\pi_p^1)\right) \otimes 
\rho_2\left(t(\pi_p^2)\right) \, p^{-s}\right) \Big\}^{-1}.
\end{equation}
Then the local factors $L_p(X,\pi_p,\rho_3)$ are:
\begin{eqnarray*}
& &( 1- \alpha_0 \beta_0 X)
( 1- \alpha_0 \beta_0 \beta_1 X)
( 1- \alpha_0 \beta_0 \beta_2 X)
( 1- \alpha_0 \beta_0 \beta_1 \beta_2 X)\\
& &( 1- \alpha_0 \alpha_1  \beta_0 X)
( 1- \alpha_0 \alpha_1  \beta_0 \beta_1 X)
( 1- \alpha_0 \alpha_1  \beta_0 \beta_2 X)
( 1- \alpha_0 \alpha_1  \beta_0 \beta_1 \beta_2 X).
\end{eqnarray*}
Since $\pi$ is associated with a Siegel modular form with trivial central character
one can employ the Langlands-Shahidi method to obtain the meromorphic continuation of 
$L(s,\pi,\rho_3)$. This works since in this case the problem can be reduced to the theory of
Euler products applied to a Levi subgroup of the exceptional group of type $F_4$.
This has been demonstrated in \cite{AS01}.
To get the functional equation one has to use a different method.
Surprisingly there also exists a Rankin-Selberg integral which involves the L-function
$L(s, \pi' \otimes \pi'',\rho_1 \otimes \rho_2)$ (if we assume the technical restriction $\pi''$ primitive)
for all even weights $k_1,k_2$. This has been discovered in \cite{He99} and generalized in \cite{BH00}.
Once such an integral representation has been found one can apply advanced techniques to
get the desired properties of the L-function. In this case the functional equation of an Eisenstein series
of Siegel type of degree $5$ leads to the meromorphic continuation and functional equation of the
L-function $L(s,\pi, \rho_3)$ we are interested in.
Let $h \in S_{k_1}$ and $G \in S_{k_2}^2$ be Hecke eigenforms with
($ 0 < k_1 \leq k_2$) even integers. For technical reasons we assume that the first Fourier-Jacobi coefficient of $G$
is non-trivial. Then the function
\begin{eqnarray*}
\mathcal{L}(s) & := &
2^{-3}\,\,
\left( 2 \pi \right)^{4-2k_2-k_1} \, 
\Gamma_{\C}(s) \,\, \Gamma_{\C}(s-k_2+1) \,\, \Gamma_{\C}(s-k_2+2) \,\,\Gamma_{\C}(s-k_1+1) 
\\ & &
\prod_{p} L_p\left( p^{-s}, \pi_h \otimes \pi_G,\rho_1 \otimes \rho_2\right)
\end{eqnarray*}
has a meromorphic continuation to the whole complex plane and satisfies the functional equation
\begin{equation}
s \mapsto k_1 + 2 k_2 -3 -s.
\end{equation}
Let $k:=k_1+2 = k_2=k_3$. Then
\begin{equation}
\Gamma_{\C}(s) \, \prod_{m=1}^3 \, \Gamma_{\C}(s-k+m) \,\, L(s, \pi' \times \pi'', \rho_1 \otimes \rho_2).
\end{equation}
has the functional equation $s \mapsto 3k - 5-s$. Hence the theorem is proven.
\end{proof}

\section{On Deligne's conjecture}
Let $F_{\pi} \in S_k^n$ be a Hecke eigenform attached to $\pi \in \mathcal{A}(GSp(2n)(\A)_k$.
If $k>2n$ be even then $F_{\pi}$ can be normalized such that the Fourier coefficients
of $F_{\pi}$ are contained in the totally real number field $E$ generated by
eigenvalues. Here we are interested the cases $n=1,2,3$.
For $n=1$ this is obtained by choosing the eigenform to be primitive, i.e 
the first Fourier coefficient
$a_g(1)=1$. For Siegel modular forms there is 
no such choice. Hence we make the assumption that the first Fourier coefficient 
in degree $2$ is non-trivial and denote such forms also primitive.
Let further $\parallel \! F_{\pi} \! \parallel$ denote the norm of $F_{\pi}$,
related to the Petersson scalar product. If for $\pi \in \mathcal{A}(GSp(4))_k$ any
primitive $F$ exists we denote $\pi$ primitive.
Finally using all the observations and results of this paper we obtain
our main result:


\begin{theorem}
Let $\widetilde{\pi}= \pi' \boxtimes \pi'' \in \mathcal{A}(\text{GSp}(6))_k$ for $\pi' \in \mathcal{A}(\text{GL}(2))_{k-2}$ 
and $\pi'' \in \mathcal{A}(\text{GSp}(4))_{k}$ primitive. 
Let $F_{\widetilde{\pi}} \in S_k^3$ be normalized.
Then the critical integers of the spinor L-function of $F_{\widetilde{\pi}}$ are
exactly $m=k,k+1,\ldots, 2k-5$.
Moreover there exists a positive $\Omega \in \R$ such that
\begin{equation}
\frac{L(m, F_{\widetilde{\pi}},\rho_3)}{{\pi}^{4m-3k+6} \Omega} \in E.
\end{equation}
\end{theorem}
%
%
\begin{proof}
For the readers convenience we recall the basic steps.
All the ingredients have already been prepared in this paper.
Let $F_{\widetilde{\pi}}$ be given.
We have proven the the full Andrianov conjecture for $L(s,\widetilde{\pi},\rho_3$.
Hence we have a meromorphic continuation on the whole complex plane.
After completion we get a functional equation. From the explicit form
of the $\Gamma$-factors we can deduce the critical values in the sense of
Deligne and obtain exactly $m=k,k+1,\ldots, 2k-5$.
With the functoriality property determined in this paper
the lifting is directly related to another L-function.
It can be identified with the so-called $GL(2)$- twisted spinor L-function.
And hence properties of this L-function can be transfered
to $L(s,\widetilde{\pi},\rho_3)$.
Let $$\Omega:= \parallel h_{\pi'}\parallel^2 \cdot \parallel G_{\pi''} \parallel^2.$$
Here $h_{\pi'}$ and $G_{\pi''}$ are primitive forms as described above.
Then finally form the arithmetic results given in \cite{BH06} we 
obtain the desired result of the theorem.
\end{proof}
Identifying the period $\Omega$ with the period given in Deligne's conjecture
$c^{\pm}(\mathcal{M})$ of the attached motive leads to a proof of
this conjecture for the lifts considered in this paper.
This would imply that the 
%
the dimension of the 
$+$ and $-$ eigenspace of the Betti realization of $\mathcal{M}$ are equal.
Finally we would like to remark that Miyawaki's conjecture of TypII \cite{Mi92}, \cite{He07e}
extended to special values result is contained in this theorem.
Moreover the assumption to be primitive can be removed.


\end{document}